\documentclass[11pt]{article}

\usepackage{times}
\usepackage{amsfonts}
\usepackage{amsthm}
\usepackage{amsmath}
\usepackage{tikz}
\usetikzlibrary{snakes}

\newcommand{\bt}{\mathbf{t}}
\newcommand{\bw}{\omega}
\newcommand{\bx}{\mathbf{x}}

\newcommand{\abs}[1]{\bigl\lvert #1 \bigr \rvert }
\newcommand{\nats}{{\mathbb N}}
\newcommand{\ints}{{\mathbb Z}}

\newcommand{\abp}{\rho^{\mbox{ab}}}
\newcommand{\p}{\rho}

\def\tribo{ {\bf t }}

\newtheorem{theorem}{Theorem}[section]
\newtheorem{lemma}[theorem]{Lemma}
\newtheorem{corollary}[theorem]{Corollary}

\theoremstyle{definition}	 
\newtheorem{definition}[theorem]{Definition}

\newtheorem{remark}[theorem]{Remark}

\newcommand{\ignore}[1]{}

\begin{document}

\title{Abelian Properties of Words\\{\normalsize (Extended abstract)}}
\author{Gw{\'e}na{\"e}l Richomme%
\footnote{Universit\'e de Picardie Jules Verne, Laboratoire MIS (Mod\'elisation, Information, Syst\`emes), 33, Rue Saint Leu, F-80039 Amiens cedex 1, FRANCE,
E-mail: \texttt{gwenael.richomme@u-picardie.fr}}
\and 
Kalle Saari%
\footnote{Department of Mathematics,
University of Turku,
FI-20014, Finland,
email: \texttt{kasaar@utu.fi}}
\footnote{Work supported by the Finnish Academy under grant 8206039.}
\and
Luca Q. Zamboni%
\footnote{
Universit\'e de Lyon, Universit\'e Lyon 1, CNRS UMR 5208 Institut Camille Jordan, B\^atiment du
Doyen Jean Braconnier, 43, blvd du 11 novembre 1918, F-69622 Villeurbanne Cedex, France, email: \texttt{zamboni@math.univ-lyon1.fr}
\indent Reykjavik University, School of Computer Science, Kringlan 1, 103 Reykjavik, Iceland,
email: \texttt{lqz@ru.is}}
\footnote{Work partially supported by grant no. 090038011 from the Icelandic Research Fund}
}

\maketitle

\abstract{We say that two finite words $u$ and $v$ are abelian equivalent if and only if they have the same number of occurrences of each letter, or equivalently if they define the same Parikh vector. In this paper we investigate various abelian properties of words including abelian complexity, and abelian powers.
We study the abelian complexity of the Thue-Morse word and the Tribonacci word, and answer an old question of G.~Rauzy by exhibiting a class of words whose abelian complexity is everywhere equal to $3.$ 
We also investigate abelian repetitions in words and show that any infinite word with bounded abelian complexity contains abelian $k$-powers for every positive integer $k$. 
}

\section{\label{overview}Introduction}

It appears that very little is known on the abelian complexity of an infinite word~\cite{CovHed,KaboreTapsoba2008TIA,Rauzy2}. In fact, to the best of our knowledge, this paper may be the first time that the very notion of abelian complexity is formally defined. This abstract provides a comprehensive study of the abelian complexity of an infinite word and its connection with other well-known word combinatorial notions. As it is intended as an extended abstract, most of the proofs of the results are omitted.

We begin with a brief introduction outlining the key definitions relevant to the paper. We assume a certain familiarity with the basic notions in Combinatorics on Words.
In Section 3 we provide extremal values for the abelian complexity. In Section 4 we
discuss a fundamental link between abelian complexity and balance of an infinite word, recalling in particular results concerning Sturmian words. 
Then in Section 5 we provide two answers to an old question of Rauzy by exhibiting two different classes of words whose abelian complexity is everywhere equal to $3.$ Sections 6 and 7 are devoted to the study of the abelian complexity of the Thue-Morse word and the
Tribonacci word. In Section~8 we investigate a connection between abelian complexity
and the presence of abelian $k$-powers. In particular, as a consequence of
the well-known van der Waerden's theorem, we deduce that an infinite word having bounded abelian complexity contains an abelian $k$-power for every positive integer $k.$  Section 9 contains a detailed study of abelian powers in Sturmian words,
the Thue-Morse word and the Tribonacci word. 

%


\section{\label{s:def}Definition of the abelian complexity}

We assume the reader is familiar with basic results and notions of combinatorics on words (for
further information see, \textit{e.g.}, \cite{ChoKar1997,Lothaire2002,Pytheas2002}). Given an \textit{alphabet} $A$, that is a finite non-empty set, we denote by $A^*$, $A^\nats$ and $A^\ints$ respectively the set of finite words, the set of (right) infinite words and the set of biinfinite words over $A$. For a finite word $u =a_1a_2\ldots a_n$ with $n \geq 0$ (when $n = 0$, $u$ is the \textit{empty word} $\varepsilon$) and $a_i \in A$, $n$ is called the {\it length} of the word $u$ and denoted $|u|.$ For each $a\in A,$ let $|u|_a$  denote the number of occurrences of the letter $a$ in $u.$ Two words $u$ and $v$ in $A^*$ are said to be {\it abelian equivalent,} denoted $u\sim_{\mbox{ab}} v,$  if and only if $|u|_a=|v|_a$ for all $a\in A.$ It is readily verified that $\sim_{\mbox{ab}}$ defines an equivalence relation on $A^*.$ 


Let $\omega$ be an {\it infinite word} on the alphabet $A,$ that is
$\omega= \omega _0\omega_1\ldots$ with each $\omega_i$ in $A.$
Any finite word of the form $\omega_{i}\omega_{i+1}\cdots \omega_{i+n-1}$ (with $i \geq 0$) is called a {\it factor} of $\omega.$ 
Let ${\mathcal F}_{\omega}(n)$ denote the set of all factors of $\omega$ of length $n,$ and set $\p_{\omega}(n)=\mbox{Card}({\mathcal F}_{\omega}(n)).$ The function $\p_{\omega}:\nats \rightarrow \nats$ is called the {\it subword complexity function} of $\omega.$
Analogously we define ${\mathcal F}^{\mbox{ab}}_{\omega}(n)={\mathcal F}_{\omega}(n)/\sim_{\mbox{ab}}$ and set 
$$\abp_{\omega}(n) =\mbox{Card}({\mathcal F}^{\mbox{ab}}_{\omega}(n)).$$

\begin{definition}
The function
$\abp =\abp_{\omega} :\nats \rightarrow \nats$ which counts the number of pairwise non abelian equivalent factors of $\omega$ of length $n$ is called the {\it abelian complexity} or {\it ab-complexity} for short. 
\end{definition}

In most instances, the alphabet $A$ will consist of the numbers $\{0,1,2,\ldots,$ $k-1\}.$  In this case, for each $u\in A^*,$ we denote by
$\Psi(u)$ the {\it Parikh vector} associated to $u,$ that is 
\[\Psi(u)=(|u|_0,|u|_1,|u|_2,\ldots ,|u|_{k-1}).\]


Given an infinite word $\omega \in A^\nats$ we set
\[\Psi_{\omega}(n)=\{\Psi(u)\,|\, u \in {\mathcal F}_{\omega}(n)\}\] so that
$$\abp_{\omega}(n)=\mbox{Card}(\Psi_{\omega}(n)).$$

\section{\label{S:extremalValues}Extremal values}

A natural question concerns the possible extremal values of the abelian complexity. 
The following result due to Coven and Hedlund is a characterization of periodic words in terms of abelian complexity:

\begin{lemma}[E.M. Coven and G.A. Hedlund, {\cite[Remark 4.07]{CovHed}}]
\label{L:aperiodic}
Let $\omega\in A^\nats \cup A^{\ints}$
be a right infinite or a biinfinite word. Then  
$\omega$ is periodic  of period $p$ if and only if 
$\abp_{\omega}(p)=1$. 
\end{lemma}

The ``only if'' part  is immediate. 
The converse follows from the observation that
a non-periodic word $\omega$ must contain arbitrarily long right special factors implying $\abp_\omega(n) \geq 2$ for all $n \geq 1$. (Let us recall that a word $u$ is a \textit{right} \textit{special factor} of an infinite word $\omega$ if for two different letters $\alpha$ and $\beta$, 
the words $u\alpha$ and $u\beta$  are both factors of $\omega$.) 

Lemma 3.1 may be regarded as the abelian analogue of the celebrated  result of M. Morse,
G.A. Hedlund \cite{MorHed1940} stating that a biinfinite word is periodic if and only if its subword complexity is bounded. Hence both  the subword complexity and the ab-complexity may be used to characterize non-periodic biinfinite words. The situation for right infinite words is slightly different since in this case infinite words with bounded complexity correspond to \textit{ultimately periodic} words (that is words of the form $uv^\infty$ where $v^\infty$ denotes the \textit{periodic} word with \textit{period} $|v|$ obtained concatenating infinitely often $v$). In the rest of the paper, we will state results only concerning right infinite words 
although many of these results remain true in the context of biinfinite words.





\medskip

Concerning the maximal abelian complexity, it is clear that it is reached by any infinite word containing all finite words as factors, as for instance the \emph{Champernowne word}
(which is obtained by concatenating all finite words enumerated with  respect to the radix order). Let us denote $\abp_{\rm max}$ the abelian complexity of such a word. Since, for any word $u$ of length $n$ over a $k$-letter alphabet, $\Psi(u)$ is a $k$-tuple $(i_1, \ldots, i_k)$ with $n = i_1 + i_2 + \cdots + i_k$, $\abp_{\rm max}$ is the maximum number of ways to write $n$ 
as the sum of $k$ nonnegative integers. This well-known number (see, \textit{e.g.}, \cite{WeissteinWeb}) is called the number of compositions of $n$ into $k$ parts and  is given by the binomial coefficient 
$ {n + k - 1 \choose  k -1}$.
This can be summarized as follows:

\begin{theorem}
For all infinite words $\bw$ over a $k$-letter alphabet, and for all $n \geq 0$,
$$
1 \leq \abp_\bw(n) \leq   {\ n+k-1 \choose k-1}.  
$$
In particular, the ab-complexity is bounded by $O(n^k)$.
\end{theorem}




We end this section with two examples illustrating some key differences between the behavior of the subword complexity and the abelian complexity.
The first one was first pointed out to us by P. Arnoux \cite{Arnoux2008}: Let $\omega$ 
denote the morphic image of the Champernowne word 
\[{\mathcal C} =01101110010111011110001001\ldots \]
 under the Thue-Morse morphism $\mu$ defined by $0\mapsto 01$ and $1\mapsto 10$ 
 Then while $\p_{\omega}(n)$ has exponential growth, we will see in Section~\ref{S:TM} that $\abp_{\omega}(n)\leq 3$ for all $n.$ 
 
The second example is to be contrasted with the first one: There exist binary infinite words having maximal abelian complexity but  linear subword complexity. Indeed let $f$ and $g$ be the morphisms defined by $f(a) = abc$, $f(b) = bbb$, $f(c) = ccc$, $g(a) = 0 = g(c)$ and $g(b) = 1$. 
Let $\omega $ denote the fixed point of $f$ beginning in $a.$ 
Then the image of $\omega$ under $g$ is the word $$\displaystyle 0\prod_{i \geq 0} 1^{3^i}0^{3^i}.$$
It is readily verified that $\abp_\bw = \abp_{\rm max}$.
Since $w$ is an automatic sequence, it has  linear complexity (see Theorems 6.3.2 and 10.3.1 in \cite{AlloucheShallit2003book}).

%

\section{\label{sec:abSturm}Links with balance properties}

In this section we investigate a connection between abelian complexity and the notion of balance: 
Following \cite{CassFerZam} we say that an infinite word $\omega \in A^\nats$ is $C$-{\it balanced}  ($C$ a positive integer) if $||U|_a-|V|_a|\leq C$ for all $a\in A$ and all factors $U$ and $V$ of $\omega$ of equal length. A word $\omega$ is said to be {\it balanced} if it is $1$-balanced. It is easy to see that 

\begin{lemma}\label{bounded=balanced}
For a word $\bw \in A^\nats \cup A^\ints$, the function $\abp_\bw$ is bounded if and only if $\bw$ is $C$-balanced for some positive integer $C$.
\end{lemma}


%



%

\noindent 
Let us recall that \textit{Sturmian} words are precisely the binary
aperiodic balanced words, where \textit{aperiodic} means non ultimately periodic (see \cite{BerSee2002}).
As noticed by G.~Rauzy \cite{Rauzy2}, it is a consequence of the works by E.M.~Coven and G.A.~Hedlund that a word $\omega$ is Sturmian if and only if for all $n \geq 0$ the cardinality of the set $\{ |u|_1 \mid u \in {\mathcal F}_{\omega}(n)\}$ is $2$. In other words, we have the following characterization which is the earliest result we know involving the notion of abelian complexity.

\begin{theorem}[E.M. Coven, G.A. Hedlund, \cite{CovHed}]\label{T:Sturm}
Let $W$ be an aperiodic binary right infinite word. Then $W$ is balanced (i.e., $W$ is a Sturmian word) if and only if $\abp(n)=2$ for all $n\geq 1.$
\end{theorem} 

Let us note that I.~Kabor\'e and T. Tapsoba also characterized using abelian complexity the family of so-called quasi-Sturmian words by insertion 
 (a subclass of the class of infinite words over a three-letter alphabet having subword complexity $n+2$) \cite{KaboreTapsoba2008TIA}. These words defined over a ternary alphabet verify $\abp(n) = 2$ for $n\neq 0$ even and $\abp(n) = 4$ for $n \neq 1$ odd.

\section{\label{S:Rauzy}Two Answers to a Question of G. Rauzy}

Inspired by the characterization of Sturmian words of Theorem~\ref{T:Sturm}, G.~Rauzy asked whether there exist aperiodic words on a $3$-letter alphabet such that $\abp(n)=3$ for all $n\geq 1$. 
%
Let $p \geq 3$ be any integer, let $\omega'$ be any Sturmian word over $\{0,1\}$ and let 
$\omega = (p-1)(p-2)\cdots 2\omega'$ ($\omega$ is written over the alphabet $\{0, 1, \ldots, (p-1)\}$. As a consequence of Theorem~\ref{T:Sturm}, we can see that $\abp_\omega(n) = p$ for all $n \geq 1$ (in particular when $p = 3$). This provides a first answer to G.~Rauzy's question. Nevertheless it is not completely satisfactory since $w$ is not recurrent (an infinite word is \textit{recurrent} if each of its factors occur infinitely often in $w$). 
We end this section by exhibiting two families of uniformly recurrent words whose ab-complexity is everywhere equal to~$3$.
Next results will provide answers including uniformly recurrent word
(let us recall that an infinite word is \textit{uniformly recurrent} if each of its factors occurs infinitely often with bounded gaps).
The first one generalizes partially Theorem~\ref{T:Sturm}.

\begin{theorem}\label{T:answer Rauzy 1} Let $\omega $ be an aperiodic balanced word on a $3$-letter alphabet. Then the ab-complexity $\abp_{\omega}(n)=3$ for all $n\geq 1.$
\end{theorem}

Theorem 5.1 is a consequence of a characterization of aperiodic balanced words due to P. Hubert in~\cite{Hu}.
%

The next theorem illustrates that the converse of Theorem~\ref{T:answer Rauzy 1} does not hold:

\begin{theorem}\label{T:answer Rauzy 2} Let $\omega'\in \{0,1\}^\nats$ be any aperiodic infinite word, and let $\omega$ be the image of $\omega'$ under the morphism $f$ defined by
$0\mapsto  012,$ and $ 1\mapsto  021.$ Then $\abp_{\omega}(n)=3$ for all $n\geq 1.$
\end{theorem}

 It would be interesting to find a characterization of all recurrent infinite words with constant ab-complexity equal to $3.$  A related question is to find
a recurrent word with constant ab-complexity equal to $4.$ We suspect no such word exists. Using~\cite{Hu}, it can be shown that there does not exist a recurrent balanced
word with constant ab-complexity equal to $4.$

\section{\label{S:TM}The ab-complexity of the Thue-Morse word}

In Section~\ref{S:extremalValues}, we announced that the image of the Champernowne word under the Thue-Morse morphism~$\mu$ has an ab-complexity bounded by $3$. More generally:

\begin{theorem}\label{T:caracAbCompTM} 
The abelian complexity of an aperiodic word $\bw \in \{0,1\}^\nats$ is 
$$\left\{
\begin{tabular}{l}
$\abp_\bw(n) = 2$ for $n$ odd,\\
$\abp_\bw(n) = 3$ for $n \neq 0$ even,\\
\end{tabular}\right.$$
if and only if there exists a word $\bw'$ such that $\bw = \mu(\bw')$, $\bw = 0\mu(\bw')$ or $\bw = 1\mu(\bw')$.
\end{theorem}

As a direct consequence we get the ab-complexity of the Thue-Morse word $\mathbf{TM}_0$, 
the fixed point of~$\mu$  beginning in $0$. 

\begin{theorem} \label{T:Thue-Morse ab-complexity} 
$\abp_{\mathbf{TM}_0}(n)=2$ for $n$ odd and $\abp_{\mathbf{TM}_0}(n)=3$ for $n \neq 0$ even.
\end{theorem}

It is quite remarkable that the previous result 
follows only from
the action of the Thue-Morse morphism. Let us note that a similar situation holds when considering the proof of Theorem~\ref{T:answer Rauzy 2}.

Note also that Theorem~\ref{T:caracAbCompTM} characterizes the class of all words having the same abelian complexity as the Thue-Morse word. 
It is known~(\cite{AbeBrl2}) that every recurrent infinite word  $\omega \in \{0,1\}^\nats$ whose subword complexity is equal to that of $\mathbf{TM}_0$ is either in the shift orbit closure of $\mathbf{TM}_0$ or is in the shift orbit closure of $\Delta (\mathbf{TM}_0)$ where $\Delta$ is the letter doubling morphism defined by $0\mapsto 00$ and $1\mapsto 11.$ 
As a consequence we deduce that:

\begin{corollary}\label{C:More on Thue-Morse} A binary infinite word has the same subword complexity and ab-complexity as the Thue-Morse word if and only if it is in the shift orbit closure of $\mathbf{TM}_0.$
\end{corollary}

To end this section, let us observe that Theorem~\ref{T:caracAbCompTM} is false when the aperiodicity hypothesis is removed. 
Indeed observe first that the word $(01)^\infty = \mu(0^\infty)$ does not have the same abelian complexity as the Thue-Morse word. 
Secondly, we let the reader verify that the ultimately periodic word $0110(1001)^\infty$ has 
the same abelian complexity as $\mathbf{TM}_0$.

\section{\label{S:tribo}The ab-complexity of the Tribonacci word}

As we have already seen, the notion of abelian complexity has many links 
to the notion of $C$-balance. 
The results of the previous section
are partially due to the fact that the image under the Thue-Morse morphism of any recurrent infinite word is $2$-balanced. 
We now investigate the ab-complexity of another well-known $2$-balanced word, the so-called 
  {\it Tribonacci word}
 \[\tribo = \tau^{\omega}(0)= 01020100102010 \cdots\]
defined as the unique fixed point of the morphism $\tau$
\[0\mapsto 01\,\,\,\,\,\,1\mapsto 02\,\,\,\,\,\,2\mapsto 0.\]

\begin{theorem}\label{triboabcom}
Let $\abp_{\tribo}$ denote the ab-complexity of the Tribonacci word $\tribo.$ Then, $\abp_{\tribo} (n) \in \{3,4,5,6,7\}$ for every positive integer $n.$
Moreover, each of these five values is assumed.
\end{theorem}

\begin{proof}
It is well-known that for all $n \geq 1$, $\tribo$ has exactly one right special factor of length $n-1$, and that, for this special factor that we denote $\tribo^<_{n-1}$, the three words $\tribo^<_{n-1}0, \tribo^<_{n-1}1,$ and $\tribo^<_{n-1}2$ are each factors of $\tribo$ of length $n$. 
Define non-negative integers $i,j,k$ by
$\Psi(\tribo^<_{n-1})=(i,j,k)$. Setting 
$${\rm Central}(n) = \{(i+1,j,k), (i,j+1,k),(i,j,k+1)\}$$
we have 
\begin{equation}\label{suBIS}{\rm Central}(n) \subseteq \Psi_{\tribo}(n).\end{equation}

Given a vector $\overrightarrow{v} = (\alpha, \beta, \gamma)$, let denote $||\overrightarrow{v}|| = 
\max(|\alpha|, |\beta|, |\gamma|)$. Observe that the set of vectors $\overrightarrow{v}$ such that $||\overrightarrow{v}  - \overrightarrow{u}|| \leq 2$ for all $\overrightarrow{u}$ in ${\rm Central}(n)$ is described by the graph of Figure~\ref{graph}    (where vectors are vertices of the graph, and each edge $(\overrightarrow{u}, \overrightarrow{v})$ denotes the fact that  $||\overrightarrow{v}  - \overrightarrow{u}|| = 1$).

Since $\tribo$ is 2-balanced, $\Psi_\tribo(n)$ is a subset of this set of twelve vectors. Moreover for the same reason, we should have $||\overrightarrow{v}  - \overrightarrow{u}|| \leq 2$ for all $\overrightarrow{u}$, $\overrightarrow{v}$ in $\Psi_\tribo(n)$. This implies that the only possibility for $\Psi_\tribo(n)$ is to be a subset of one of the three sets delimited by a regular hexagon in Figure~\ref{graph}, or one of the three sets delimited by an equilateral triangle of base length 2. These sets have cardinalities 7 and 6 respectively showing that $\abp_\omega(n) \leq 7$.

\begin{figure}\begin{center}
\includegraphics[height=6cm]{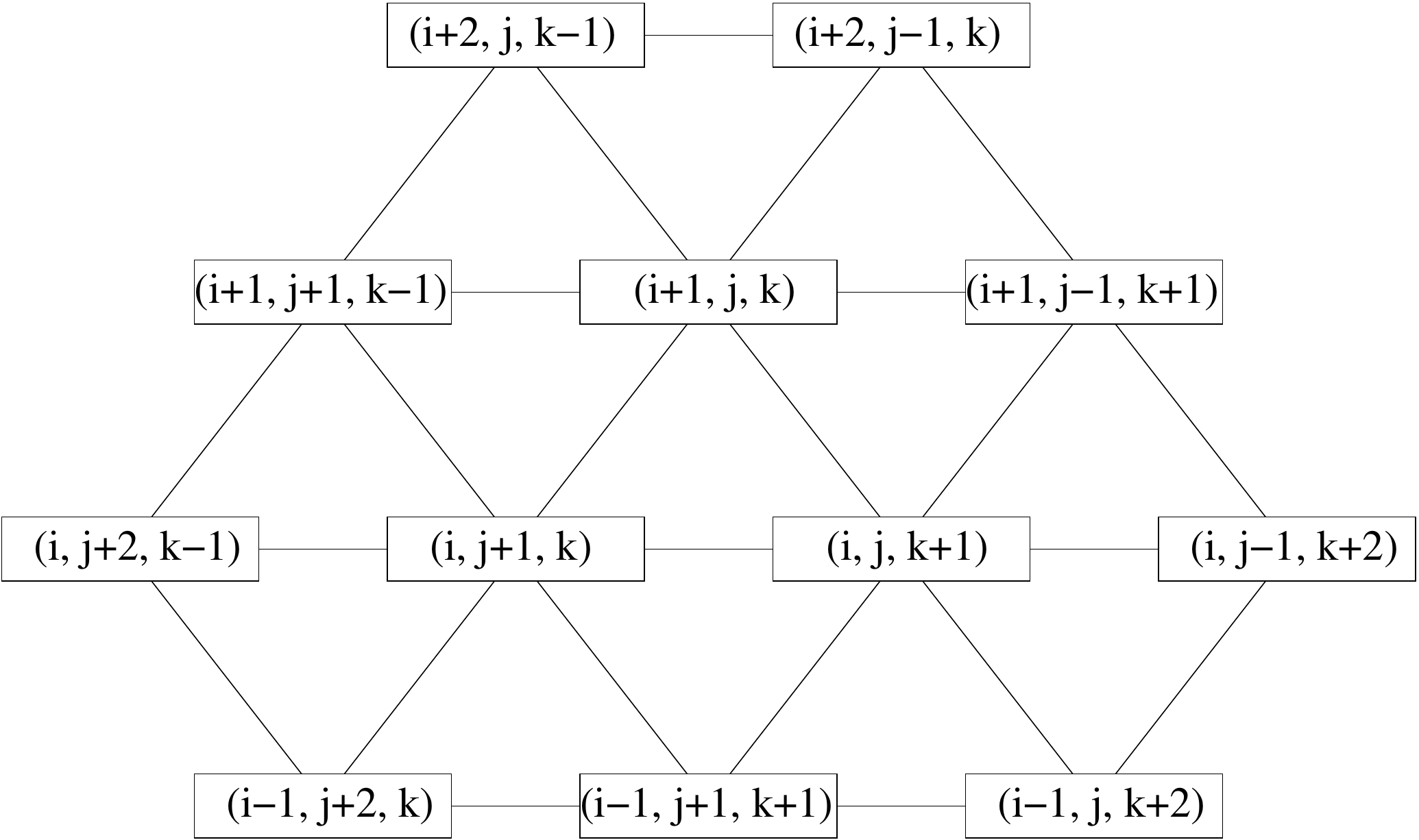}
\caption{\label{graph}Links between Parikh vectors}\end{center}
\end{figure}

\medskip

By computer simulation we find that
\[(\abp(n))_{n\geq 1}=3 3 4 3 4 4 4 3 4 4 4 4 4 4 3 4 4 4 4 4 4 4 4 4 4 4 4 3 4 5 5 4 4 4 4 4 5 5 4 4 4 4\ldots\]
In particular, the least $n$ for which $\abp(n)=5$ is for $n=30.$ We also found that the smallest $n$ for which $\abp(n)=6$ is
$n=342,$ and the smallest $n$ for which  $\abp(n)=7$ is $n=3914.$ The next four values of $n$ for which $\abp(n)=7$ are $n= 4063, 4841,  4990, 7199.$ 
\end{proof}

It is surprising to us that the value $\abp_{\tribo}(n)=7$ does not occur until $n=3914,$ but then re-occurs relatively shortly thereafter. There are many interesting and mysterious properties observed in the behavior of the Tribonacci word: For instance, for all $n\leq 184,$ if $U$ and $V$ are factors of $\tribo$ of length $n,$ with $U$ a prefix of $\tribo,$ then $||U|_a-|V|_a|\leq 1,$ for all $a\in \{0,1,2\}.$ But then this fails for $n=185.$ The intererested reader will find in \cite{RichommeSaariZamboni2009tribo} a proof  that $\abp_{\tribo}(n)=3$ if and only if $\tribo$ has a bispecial factor of length $n-1.$ It is also proved in this paper that the abelian complexity of $\tribo$ attains the value $7$ infinitely often. It is an open question to find a proof that values $4$, $5$ and $6$ are also attained infinitely often.



To end this section, we would like to stress the importance of the $2$-balance of the Tribonacci word to prove Theorem~\ref{triboabcom}. Although this result is cited in numerous articles, 
we were unable to find a proof of this fact in the literature. We wrote a combinatorial proof in \cite{RichommeSaariZamboni2009tribo}. We also have a proof of this fact that uses the spectral properties of the adjacency matrix associated to the generating morphism~\cite{RicSaaZam2009prep}.

\section{\label{s:ab-powers}Links with abelian powers}

We now consider abelian powers.
Repetitions occurring in an infinite word is a topic of great interest having applications to a broad range of areas (see, \textit{e.g.}, \cite{AB, AlDaQuZa, AlZa,Lothaire2005}). One stream of research dating back to the works of Thue \cite{Thue1906,Thue1912} is the study of patterns avoidable by infinite words (see, \textit{e.g.}, \cite{Lothaire1983book,Lothaire2002,RichommeSeebold2004IJFCS,RichommeWlazinski2007DAM,BerstelLauveReutenauerSaliola2008book}). 
In the abelian context, F.M.~Dekking \cite{Dekking1979} showed that abelian 4-powers are avoidable on a 2-letter alphabet and that abelian cubes are avoidable on a 3-letter alphabet. V.~Ker\"anen \cite{Keranen1992ICALP} proved that abelian squares are avoidable on four letters. 
Let us recall that an \textit{abelian} $k${\it-power} is any non-empty word on the  form $W=U_1U_2\cdots U_k$ where $U_i\sim_{\mbox{ab}}U_j$ for all $1\leq i,j\leq k.$

\begin{theorem}\label{bounded} 
Any infinite word having bounded abelian complexity contains an abelian $k$-power for every positive integer $k.$
\end{theorem}

This theorem could be considered as an abelian analogue of the celebrated result by M.~Morse and G.A.~Hedlund: ``\textit{infinite words with bounded subword complexities are ultimately periodic}". 
The proof of Theorem 8.1 uses van der Waerden's theorem.

Theorem~\ref{bounded} raises  natural questions:
Is it  true that any recurrent infinite word having a bounded abelian complexity has the property that each position begins in an abelian $k$-power?
What about the case of uniformly recurrent word?
Note that here  the requirement that the abelian complexity be bounded is important. Indeed, in \cite{Dekking1979} F.M.~Dekking showed that the fixed point of the morphism $0 \mapsto 011$ and $1 \mapsto 0001$ is  abelian 4-power free (this word is recurrent since the morphism is primitive).

This problem seems  difficult since we do not know the answer even in the special case of the Thue-Morse word.

\section{\label{sec:ab-repet-sturm}Abelian repetitions in Sturmian words}

When considering stronger hypothesis than in Theorem~\ref{bounded}, one can naturally expect to have a stronger result. This is what happens in Theorem~\ref{T:ab-k pow in Sturm} below dealing with Sturmian words, that is by Theorem~\ref{T:Sturm}, words having abelian complexity $2$ everywhere: 

\begin{theorem}[\cite{RichommeSaariZamboni}]\label{T:ab-k pow in Sturm} 
For every Sturmian word $\omega$ and every integer $k\geq 1$, there exist two integers $\ell_1$ and $\ell_2$ such that each position in $\omega$ has an occurrence of an abelian $k$-power with abelian period $\ell_1$ or $\ell_2.$
\end{theorem}

\noindent 
Note that the situation in Theorem~\ref{T:ab-k pow in Sturm} is different than for usual powers. Indeed every Sturmian word begins with infinitely many square, but not necessarily with a cube \cite{BerHolZam}.
Sturmian words are optimal in the following sense: 
In the next theorem, we say that $w$ has abelian period $p$ if $w=x_1 x_2 \cdots x_n$ for some  pairwise abelian equivalent words $x_i$ with
$p=\abs{x_1}=\cdots=\abs{x_n}$.

\begin{theorem}\label{T:optSturm} If an infinite word $\bx$ is abelian $k$-repetitive such that every position starts with an abelian $k$-power with a fixed abelian period $m$, then $\bx$ is ultimately periodic.
\end{theorem}
%



\begin{remark}
The property mentioned in Theorem~\ref{T:ab-k pow in Sturm} is  not characteristic for Sturmian words. Indeed if $\bt$ is a Sturmian word and $f$ is the morphism defined by $f(a) = aa$, $f(b) = ab$, the word $f(\bt)$ is not Sturmian and verifies this property. More precisely if every position of $\bt$ starts with an abelian $k$-power of abelian period either $\ell_1$ or $\ell_2$, then every position of $f(\bt)$ starts with an abelian $k$-power of abelian period either $2\ell_1$ or $2\ell_2$. Considering instead of $f$ the morphism $g$ defined by $g(a) = c_1c_2\cdots c_na$ and $g(b) =c_1c_2\cdots c_nb$  with $c_1, \ldots, c_n$ letters, one can find non ultimately periodic words over arbitrary alphabet having the previous property. 
\end{remark}

We end this section with two further results on abelian repetitions:

\begin{theorem}\label{T:prefix2} For all integers $k\geq 1$, each suffix of the Tribonacci word
begins in an infinite number of abelian $k$-powers.
\end{theorem}

We do not know whether this holds for all words in the subshift generated by
the Tribonacci word. For the Thue-Morse word, we have

\begin{theorem}
Each suffix of the Thue-Morse word begins in an abelian $6$-power.
\end{theorem}

We  do not know whether this holds for abelian $7$-powers.


\bibliographystyle{plain}
\bibliography{RSZ}
 
\end{document}